\documentclass[a4paper,11pt]{article}
\usepackage[utf8]{inputenc}
\usepackage{amsmath}
\usepackage{amssymb}
\usepackage{amsthm}
\usepackage{changebar}
\usepackage{enumerate}

\newcommand{\ud}{\mathrm{d}}
\newcommand{\pd}{\partial}
\newcommand{\e}{\varepsilon}
\newcommand{\eps}{\varepsilon}

\newcommand{\bK}{\mathbb{K}}
\newcommand{\bN}{\mathbb{N}}
\newcommand{\bR}{\mathbb{R}}

\newcommand{\R}{\bR}
\newcommand{\N}{\bN}

\newcommand{\cA}{\mathcal{A}}

\newcommand{\cC}{\mathcal{C}}

\newcommand{\cE}{\mathcal{E}}

\newcommand{\cG}{\mathcal{G}}

\newcommand{\cN}{\mathcal{N}}

\newcommand{\csub}{\subset \subset}
\newcommand{\coleq}{\mathrel{\mathop:}=}

\newcommand{\fX}{\mathfrak{X}}

\DeclareMathOperator{\id}{id}

\DeclareMathOperator{\supp}{supp}

\DeclareMathOperator{\pr}{pr}

\DeclareMathOperator{\Hom}{Hom}

\newcommand{\tang}{\mathrm{T}}
\newcommand{\Cinf}{C^\infty}

\newcommand{\Lin}{\mathrm{L}}

\providecommand{\norm}[1]{\left\lVert#1\right\rVert}
\providecommand{\abso}[1]{\left\lvert#1\right\rvert}

\newcommand{\sk}[2]{\widetilde\cA_{#1}(#2)}
\newcommand{\lsk}[2]{\widetilde\cA_{#1}(#2)}
\newcommand{\ub}[2]{\hat\cA_{#1}(#2)}
\newcommand{\lub}[2]{\cA_{#1}(#2)}

\newcommand{\comp}{\subset\subset}

\title{Manifold-valued generalized functions in full Colombeau spaces}
\author{Michael Kunzinger\footnote{University of Vienna, Faculty of Mathematics,
Nordbergstr.\ 15, A-1090 Vienna, Austria, michael.kunzinger@univie.ac.at}, 
Eduard Nigsch\footnote{University of Vienna, Faculty of Mathematics,
Nordbergstr.\ 15, A-1090 Vienna, Austria, eduard.nigsch@univie.ac.at}}

\newtheorem{theorem}{Theorem}

\newtheorem{proposition}[theorem]{Proposition}
\newtheorem{lemma}[theorem]{Lemma}
\newtheorem{definition}[theorem]{Definition}
\newtheorem{remark}[theorem]{Remark}

\newcommand{\cD}{\mathcal{D}}

\newcommand{\mvb}[2]{\hat\cE[#1,#2]} 
\newcommand{\mvm}[2]{\hat\cE_M[#1,#2]} 
\newcommand{\mvq}[2]{\hat\cG[#1,#2]} 
\newcommand{\mvqt}[2]{\tilde\cG[#1,#2]} 

\newcommand{\vmb}[2]{\hat\cE^{\mathrm{VB}}[#1,#2]} 
\newcommand{\vmm}[2]{\hat\cE_M^{\mathrm{VB}}[#1,#2]} 
\newcommand{\vmq}[2]{\Hom_{\hat \cG}[#1,#2]} 
\newcommand{\vbsim}{\sim_{vb}}
\newcommand{\vbzsim}{\sim_{\mathrm{vb}0}}

\begin{document}

\date{}
\maketitle

\begin{abstract}
We introduce the notion of generalized function taking values in a smooth manifold into the
setting of full Colombeau algebras. After deriving a number of 
characterization results we also introduce a corresponding concept of generalized vector bundle homomorphisms and,
based on this, provide a definition of tangent map for such generalized functions. 

\vskip 1em



\end{abstract}

\section{Introduction}

When studying geometrical problems in the presence of singularities, linear distributional geometry (e.g., \cite{M,P}) has 
a number of natural limitations, in particular concerning nonlinear operations (tensor calculus, curvature).
For this reason, nonlinear theories of generalized functions based on Colombeau's construction (\cite{C1,C2,GKOS,O})
have been extended by various authors (e.g., \cite{DP,DD,GKOS,KS,V}) to a nonlinear distributional geometry capable of extending
the distributional approach to certain nonlinear situations, in particular in the setting of pseudo-Riemannian 
geometry. 

A major obstacle for modeling geometrical objects like, e.g., flows of singular vector fields or geodesics of
distributional space-time metrics in linear distribution theory is the absence of a concept of generalized
function taking values in a differentiable manifold. In the special Colombeau setting, this problem was addressed
in \cite{kun,kun2}; the resulting theory has found a number of applications both in the theory of 
generalized functions and in mathematical physics (cf.\ \cite{EG,GKOS,sheaves}).

While special Colombeau algebras have been successfully applied to many situations where a natural way 
of regularizing is available, they do not in general possess a canonical embedding of the space of distributions. 
In many applications, particularly in General Relativity, however, it is desirable to work in the setting of
such a canonical embedding. Indeed, the guiding principle of General Relativity is coordinate invariance (general 
covariance), so it appears natural to also consider covariant regularization procedures when modelling singularities
in this context. Without a canonical embedding of distributions, in order to obtain a covariant result it is necessary
to explicitly check coordinate invariance of the results thus achieved. This has been done in a number of cases, 
most prominently in the calculation of the distributional curvature of cosmic strings, see \cite{CVW,VW}.
Built-in coordinate independence of the entire construction allows to avoid this additional step. For an
in-depth discussion we refer to \cite{SV}. 

While already introduced in \cite{C1,C2}, the full setting originally 
was not diffeomorphism invariant, hence did not lend itself to applications
in geometry. Over the past 15 years, however, the theory has been restructured in order to 
incorporate coordinate invariance, first in the scalar case (\cite{ColMeril,GFKS,GKSV,Jel1,Jel2})
and recently also in the tensorial setting (\cite{GKSV2,sotonTF}). So far, however,
this theory does not allow to consider generalized functions taking values in smooth manifolds.
The present article supplies the necessary constructions to fill this gap.

The plan of the paper is as follows. In Section \ref{PN} we introduce some basic notations and recall
those parts of the local and global theory of full Colombeau algebras necessary for our approach.
In Section \ref{MF} we introduce manifold-valued generalized functions in this context. We characterize
moderateness and equivalence for these generalized functions, the basic idea being to reduce these
properties to the corresponding ones of scalar valued Colombeau functions via composition with
smooth functions on the target manifold. Finally, in Section \ref{Hom} we introduce generalized 
vector bundle homomorphisms and give analogous characterizations. As a main example we 
define the tangent map of any manifold-valued Colombeau generalized function.

\section{Preliminaries and Notation}\label{PN}

Throughout this article the letters $X$ and $Y$ will represent smooth paracompact Hausdorff manifolds of dimensions $\dim X=n$ and $\dim Y=m$. A vector bundle $E$ over $X$ with projection $\pi_E$ will be denoted by $\pi_E: E \to X$, as typical vector bundles we will use $\pi_E: E \to X$ and $\pi_F: F \to Y$ with dimensions $n'$ and $m'$ of the fibers, respectively. $\Hom(E,F)$ ($\Hom_c(E,F)$) denotes the space of (compactly supported) vector bundle homomorphisms from $E$ to $F$. $\Omega^n_c(U)$ is the space of compactly supported $n$-forms on $U$, an open subset of a manifold or of $\R^n$. For any open set $U \subseteq \bR^n$ we define $\hat\lambda: \Omega^n(U) \to \Cinf(U)$ to be the linear isomorphism assigning to an $n$-form $\omega \in \Omega^n_c(U)$ the smooth function $x \mapsto \omega(x)(e_1,\dotsc,e_n)$ on $U$, where $\{\,e_1,\dotsc,e_n\,\}$ is the Euclidean basis in $\bR^n$. Generally, our background reference for differential geometry is \cite{marsden}. Calculus of smooth functions on infinite-dimensional locally convex spaces is understood in the sense of the so-called convenient calculus of \cite{KM}. $B_r(x)$ for $r>0$ and $x \in \R^n$ denotes the open ball of radius $r$ around $x$ in $\R^n$, $\pr_i$ is the projection of a product onto the $i$th factor. Finally, we assume the reader to be familiar with the local and global full diffeomorphism invariant Colombeau algebras $\cG^d(\Omega)$ and $\hat\cG(X)$ and the corresponding symbols for basic spaces and subspaces of moderate and negligible functions ($\cE^d(\Omega)$, $\cE_M^d(\Omega)$, $\cN^d(\Omega)$; $\hat\cE(X)$, $\hat\cE_M(X)$, $\hat\cN(X)$, cf.~\cite{GKSV}), as well as the spaces $\lub q\Omega$ and $\ub qX$ which are used in the construction.

Given a mapping $R \in C^\infty(\ub0X \times X, Y)$ and charts $(V,\varphi)$ in $X$ and $(W,\psi)$ in $Y$, we define the local expression of $R$ with respect to these charts as
\[ R_{W,V} \coleq \psi \circ R \circ ((\varphi^* \circ \hat\lambda^{-1}) \times \varphi^{-1}). \]
This is a smooth function from $((\hat\lambda \circ \varphi_*) \times \varphi)((\ub0V \times V) \cap R^{-1}(W)) \subseteq \lub0{\bR^n} \times \R^n$ into $\psi(W) \subseteq \bR^m$.

Similarly, for vector bundles $E \to X$ and $F \to Y$ we consider a mapping $s \in \Cinf(\ub0X \times E, F)$ such that for each fixed $\omega$ the mapping $s(\omega, \cdot)$ is a vector bundle homomorphism from $E$ to $F$. We define the local expression of $s$ with respect to vector bundle charts $(V, \Phi)$ in $E$ and $(W, \Psi)$ in $F$ 
over charts $\varphi$ of $X$ and $\psi$ of $Y$ as
\[ s_{W,V} \coleq \Psi \circ s \circ ((\varphi^* \circ \hat\lambda^{-1})\times \Phi^{-1}), \]
which is a smooth function from $((\hat\lambda \circ \varphi_*) \times \Phi)((\ub0V \times \pi_E^{-1}(V)) \cap s^{-1}(\pi_F^{-1}(W)))$ into $\Psi(W)$.
Because $(\pr_1 \circ s_{W,V})(\varphi, x , \xi)$ does not depend on $\xi$ it makes sense to define
\[ s_{W,V}^{(1)}(\varphi,x) \coleq (\pr_1 \circ s_{W,V})(\varphi, x, 0).\]
This is a smooth function from the set of all pairs $(\phi, x) \in \lub0{\varphi(V)} \times \varphi(V)$ satisfying $\pi_F(s(\varphi^*(\hat\lambda^{-1}(\phi)), \Phi^{-1}(x,0))) \in W$ into $\psi(W) \subseteq \bR^m$. The definition of $s_{W,V}^{(1)}$ is compatible with a change of chart; the mapping thus defined on the manifold and having $s_{W,V}^{(1)}$ as local expression shall be denoted by $\underline{s} \in \Cinf(\ub0X \times X, Y)$.

Next we note that $\pr_2 \circ s_{W,V}$ is smooth into $\R^{m'}$ and linear in the third variable. By the exponential law \cite[3.12]{KM} it corresponds to a smooth mapping denoted by $s_{W,V}^{(2)}$ from $((\hat\lambda \circ \varphi_*) \times \Phi)((\ub0V \times \pi_E^{-1}(V)) \cap s^{-1}(\pi_F^{-1}(W))$ into $\Lin(\R^{n'}, \R^{m'})$, the space of all linear mappings from $\R^{n'}$ to $\R^{m'}$. With this we can write the local expression of $s$ in the form
\[ s_{W,V}(\phi,x,\xi) = (s_{W,V}^{(1)}(\phi, x), s_{W,V}^{(2)}(\phi,x) \cdot \xi) \]
for all $(\phi, x, \xi)$ in its domain of definition.

In case the target manifold is some finite-dimensional real space we use the identity chart and simply write $R_V$ instead of $R_{W,V}$. Similarly, if $F$ is a trivial vector bundle over a finite-dimensional real space we write $s_V$, $s_V^{(1)}$ and $s_V^{(2)}$, accordingly.


The spaces of smoothing kernels $\sk q{X}$ are defined in \cite[Definition 3.3.5]{GKOS}. Their local equivalents are the spaces of local smoothing kernels $\lsk q{\Omega}$ on subsets $\Omega$ of $\R^n$ as defined in \cite[Definition 4.3]{Nig}:

\begin{definition}\label{localsmoothkern}\begin{enumerate}
\item[(1)]
A mapping $\tilde \phi \in C^\infty(I \times \Omega, \lub0\Omega)$ is called a \emph{local smoothing kernel} on $\Omega$ if it satisfies the following conditions:
\begin{enumerate}
\item[(i)] $\forall K \csub \Omega$ $\exists \e_0>0, C>0$ $\forall x \in K$ $\forall \e\le \e_0$: $\supp \tilde\phi(\e,x) \subseteq B_{\e C}(x) \subseteq \Omega$.
\item[(ii)] $\forall K \csub \Omega$ $\forall \alpha,\beta \in \bN_0^n$ the asymptotic estimate $\abso{ (\pd_y^\beta \pd_{x+y}^\alpha \tilde \phi) (\e,x)(y)}$ $=$ $O(\e^{-n-\abso{\beta}})$ holds uniformly for $x \in K$ and $y \in \Omega$.
\end{enumerate}
The space of all local smoothing kernels on $\Omega$ is denoted by $\lsk 0\Omega$.
\item[(2)] For each $k \in \bN$, $\lsk k\Omega$ is the subset of $\lsk0\Omega$ consisting of all $\tilde\phi$ such that for all $f \in C^\infty(\Omega)$ and all compact subsets $K$ of $\Omega$ the estimate $\abso{f(x) - \int_\Omega f(y)\tilde\phi(\e,x)(y)\, \ud y} = O(\e^{k+1})$ holds uniformly for $x \in K$.
\end{enumerate}
\end{definition}

On each chart $(V,\varphi)$ in $X$ there is an isomorphism $\sk q{V} \cong \lsk q{\varphi(V)}$ realized by the mapping $\Phi \mapsto \tilde\phi \coleq \hat\lambda \circ \varphi_* \circ \Phi \circ (\id \times \varphi^{-1})$.


 Using the local characterization of moderateness and negligibility established in \cite[Theorems 4.3 and 4.4]{GKSV} one immediately obtains that local smoothing kernels are suitable test objects for the local diffeomorphism invariant Colombeau algebra $\cG^d$, resulting in the following local characterization of moderateness and negligibility in $\hat\cE(X)$ (cf. \cite[Theorems 3.3.15 and 3.3.16]{GKOS}):

\begin{proposition}
 \begin{enumerate}[(i)]
  \item $R \in \hat\cE(X)$ is moderate if and only if for all charts $(V, \varphi)$ in $X$, $\forall K \csub \varphi(V)$ $\forall k \in \bN_0$ $\exists N \in \bN$ $\forall \tilde \phi \in \lsk0{\varphi(V)}$: 
$$
\sup_{x \in K}\norm{D^{(k)}(R_V(\tilde \phi(\e,x),x))} = O(\e^{-N}).
$$
\item $R \in \hat\cE_M(X)$ is negligible if and only if for all charts $(V, \varphi)$ in $X$, $\forall K \csub \Omega$ $\forall k \in \bN_0$ $\forall m \in \bN$ $\exists q \in \bN$ $\forall \tilde \phi \in \lsk q{\varphi(V)}$:
$$
\sup_{x \in K}\norm{D^{(k)}(R_V(\tilde \phi(\e,x),x))} = O(\e^m).
$$
 \end{enumerate}
\end{proposition}

\section{Manifold-valued generalized functions}\label{MF}

We begin with the following definitions of the basic space of manifold-valued generalized functions and an appropriate notion of c-boundedness that is based on the corresponding notion of the special algebra (cf.\ \cite{kun,kun2}). For the full algebra one has to replace the index set by $\ub0X$ and the quantifier ``for small $\e$'' by the appropriate asymptotics used throughout the construction of full Colombeau algebras.

\begin{definition}$\mvb{X}{Y} \coleq \Cinf (\ub0X \times X, Y)$ is the basic space of (full) Colombeau generalized functions on $X$ taking values in $Y$.
\end{definition}

An element of the basic space $\mvb{X}{Y}$ is called c-bounded if it asymptotically maps 
compact sets to compact sets, more precisely: 

\begin{definition}$R \in \mvb{X}{Y}$ is called \emph{c-bounded}  if
\begin{multline}
 \forall K \csub X\ \exists L \csub Y\ \forall \Phi \in \sk0{X}\\
 \exists \e_0>0\ \forall \e<\e_0\ \forall p \in K: R(\Phi(\e,p),p) \in L.
\end{multline}
\end{definition}
In particular, for $Y=\bR$ or $\bK$ the above gives a definition of c-boundedness for elements of $\hat\cE(X)$.

For the quotient construction we recall that a generalized function $S \in \hat \cE(X)$ is defined to be moderate (see \cite[Def.\ 3.10]{GKSV}) if
\begin{multline}
 \forall K\csub M\ \exists N\in \bN\ \forall l\in \bN_0\ \forall X_1,\dotsc,X_l \in \fX(M)\ \forall \Phi\in\sk0X:\\
 \sup_{p \in K} \abso{L_{X_1}\dotsc L_{X_l} S(\Phi(\e,p),p)} = O(\e^{-N}),
\end{multline}
where $\fX(M)$ denoted the space of smooth vector fields on $X$. In local notation this condition is equivalent to
\begin{multline}\label{scal_mod_local}
\forall \textrm{ charts }(V, \varphi)\textrm{ in }X\ \forall K \csub V\ \exists N\in \bN\ \forall \alpha\in \bN_0^n\ \forall \tilde \phi \in \lsk0{\varphi(U)}:\\
 \sup_{x \in \varphi(K)}\abso{\pd^\alpha S_U(\tilde\phi(\e,x),x)} = O(\e^{-N}).
\end{multline}

By analogy to the definition of moderateness for manifold-valued generalized functions in the special setting (\cite[Definition 2.2]{kun}) and condition \eqref{scal_mod_local} we are led to the following definition of moderateness in $\mvb{X}{Y}$.

\begin{definition}\label{mv_mod}An element $R \in \mvb{X}{Y}$ is called \emph{moderate} if
 \begin{enumerate}[(i)]
  \item $R$ is c-bounded,
  \item for all charts $(V,\varphi)$ in X and $(W,\psi)$ in Y, all $L \csub V$ and $L' \csub W$ and all $k \in \bN_0$ there exists $N \in \bN$ such that for all $\tilde\phi \in \lsk0{\varphi(V)}$ the estimate $\|D^{(k)} (R_{W,V}(\tilde\phi(\e,x),x)) \| = O(\e^{-N})$ holds uniformly for $x \in \varphi(L) \cap R_{W,V}(\tilde\phi(\e, \cdot), \cdot)^{-1}(\psi(L'))$.
 \end{enumerate}
 By $\mvm{X}{Y}$ we denote the space of all moderate elements of $\mvb{X}{Y}$.
\end{definition}

In order to obtain an equivalence relation for the quotient we adapt \cite[Definition 2.4]{kun}:

\begin{definition} \label{equivdef}
 Two elements $R,S \in \mvm{X}{Y}$ are called \emph{equivalent} (denoted by $R \sim S$) if 
\begin{enumerate}[(i)]
 \item for any Riemannian metric $h$ on $Y$, $\forall K \csub X$ $\exists N \in \bN_0$ $\forall \Phi \in \sk N{X}$: 
\begin{equation}
 \sup_{p \in K}d_h(R(\Phi(\e,p),p),S(\Phi(\e,p),p)) \to 0\quad(\e \to 0),
\end{equation}
 \item for all charts $(V,\varphi)$ in $X$ and $(W,\psi)$ in $Y$, all $L \csub V$ and $L' \csub W$ and all $k \in \bN_0$ and $m \in \bN$ there exists $N \in \bN_0$ such that for all $\tilde\phi \in \lsk N{\varphi(V)}$ we have the estimate 
$$
 \| D^{(k)}(R_{W,V}(\tilde\phi(\e,x),x)) -  D^{(k)} (S_{W,V}(\tilde\phi(\e,x),x))\| = O(\e^m)
$$ 
uniformly for $x \in \varphi(L) \cap R_{W,V}(\tilde\phi(\e, \cdot), \cdot)^{-1}(\psi(L'))
\cap S_{W,V}(\tilde\phi(\e, \cdot), \cdot)^{-1}(\psi(L'))$.
\end{enumerate}
If $R$, $S$ satisfy (i) and (ii) for $k = 0$ we call them equivalent of order $0$, written $R\sim_0 S$.
\end{definition}

\begin{remark} {\rm 
Definition \ref{equivdef} (i) is formulated with respect to the distance function $d_h$ induced by 
some Riemannian metric $h$ on $M$. Because both $R$ and $S$ are c-bounded it does not matter which Riemannian metric is chosen (cf., e.g., \cite{GKOS}, Lemma 3.2.5).
}
\end{remark}


\begin{definition}
The quotient $\mvq{X}{Y} := \mvm{X}{Y}/\sim$ is called the space of (full) Colombeau generalized
functions on $X$ taking values in $Y$.
\end{definition}

For $R\in \mvm{X}{Y}$ we denote by $[R]$ its equivalence class in $\mvq{X}{Y}$.

With these definitions one can show the analogues of \cite[Propositions 3.1 and 3.2]{kun2}:

\begin{proposition}\label{boundchar}Let $R \in \mvb{X}{Y}$. The following conditions are equivalent:
 \begin{enumerate}[(i)]
  \item $R$ is c-bounded.
  \item $f \circ R$ is c-bounded for all $f \in \Cinf(Y)$.
 \end{enumerate}
\end{proposition}

\begin{proof}(i) $\Rightarrow$ (ii) is clear.

(ii) $\Rightarrow$ (i): Let $\iota: Y \to \R^N$ be a Whitney embedding and let $K \csub X$. By assumption there are compact sets $L_i \csub \R$ such that $\forall \Phi \in \sk0X$ $\exists \e_0>0$ $\forall \e<\e_0$: $(\pr_i \circ \iota \circ R)(\Phi(\e,p),p) \in L_i$ for all $p \in K$. This implies that $(\iota \circ R)(\Phi(\e,p),p)$ is contained in a compact set for the same $\Phi$, $\e$, and $p$, from which the claim is immediate.
\end{proof}

Note that -- in contrast to the situation in the special algebra -- it does not seem to be the case that moderateness of $f \circ R$ for all $f \in \Cinf(Y)$ implies c-boundedness of $R$.

\begin{proposition} \label{modchar}
Let $R \in \mvb{X}{Y}$. The following statements are equivalent:
 \begin{enumerate}[(a)]
  \item $R \in \mvm{X}{Y}$.
  \item \begin{enumerate}[(i)]
    \item $R$ is c-bounded,
    \item $f \circ R \in \hat\cE_M(X)$ for all $f \in \cD(Y)$.
    \end{enumerate}
  \item \begin{enumerate}[(i)]
    \item $R$ is c-bounded,
    \item $f \circ R \in \hat\cE_M(X)$ for all $f \in \Cinf(Y)$.
    \end{enumerate}
 \end{enumerate}
\end{proposition}

\begin{proof}
 (c) $\Rightarrow$ (b) is clear.

 (b) $\Rightarrow$ (a): Let charts $(V, \varphi)$ in $X$ and $(W,\psi)$ in $Y$, compact subsets $L \csub V$ and $L' \csub W$, $k \in \bN_0$ and $\tilde \phi \in \lsk 0{\varphi(V))}$ be given. Choose $f \in \cD(Y)^m$ with $\supp f \subseteq W$ and $f = \psi$ in a neighborhood of $L'$ and set $f_j \coleq \pr_j \circ f$ and $\psi_j \coleq \pr_j \circ \psi$. Then for any $x \in \varphi(L \cap R_{W,V}(\tilde \phi(\e, \cdot),\cdot)^{-1}(\psi(L')))$ the equality
\begin{align*}
\psi_j \circ (\psi^{-1} \circ R_{W,V}(\tilde \phi(\e, \cdot),\cdot)) & = f_j \circ (\psi^{-1} \circ R_{W,V}(\tilde \phi(\e, \cdot), \cdot)) \\
&= (f_j \circ R)_W(\tilde \phi(\e, \cdot), \cdot)
\end{align*}
holds in a neighborhood of $x$. Because $f_j \circ R$ is moderate the result follows by differentiating.

 (a) $\Rightarrow$ (c): Let $f \in \Cinf(Y)$ and $K \csub X$. Without loss of generality we may assume $K \csub V$ for some chart $(V, \varphi)$ in $X$. Because $R$ is c-bounded we may choose $L \csub Y$ such that
\[ \forall \Phi \in \sk0X\ \exists \e_0>0\ \forall \e<\e_0\ \forall p \in K: R(\Phi(\e,p),p)) \in L. \]
We cover $L$ by charts $(W_l, \psi_l)$ in $Y$ with $1 \le l \le s$ and write $L = \bigcup_{l=}^s L_l'$ with $L_l' \csub W_l$ for each $l$. Now given any $\Phi \in \sk0X$, for all small $\e$ and each $p \in K$ there is $l$ such that $R(\Phi(\e,p),p) \in L_l'$ and thus
\[ 
(f \circ R)_V(\tilde \phi(\e,\varphi(p)),\varphi(p)) = (f \circ \psi_l^{-1}) \circ R_{W_l, V}(\tilde\phi(\e,\varphi(p)),
\varphi(p)), 
\]
where $\tilde\phi \coleq \hat\lambda \circ \varphi_* \circ \Phi \circ (\id \times \varphi^{-1})$. For any $k\in \N$, applying Definition \ref{mv_mod} to $L \coleq K$, $L' \coleq L_l'$, $(W_l, \psi_l)$ we obtain $\e_1 = \min_{1 \le l \le s}\e_1^l$, $N = \max_{1 \le l \le s}N_l$ 
such that
\[ \sup_{x \in \varphi(K)} \abso{D^{(k)} (f \circ R)_V (\tilde \phi(\e, x), x)} = O(\e^{-N}), \]
hence $f \circ R$ is moderate.
\end{proof}

The following result characterizes the equivalence relation $\sim$ in $\mvm{X}{Y}$.

\begin{theorem} \label{equivchar}
Let $R,S \in \mvm{X}{Y}$. The following statements are equivalent:
 \begin{enumerate}[(i)]
  \item $R \sim S$.
  \item $R \sim_0 S$.
  \item For every Riemannian metric $h$ on $Y$,
\begin{multline}
 \forall K \csub X\ \forall m \in \bN\ \exists N \in \bN_0\ \forall \Phi \in \sk N{X}:\\
 \sup_{p \in K}d_h(R(\Phi(\e,p),p), S(\Phi(\e,p),p)) = O(\e^m)\quad(\e \to 0).
\end{multline}
  \item $(f \circ R - f \circ S) \in \hat\cN(X)$ for all $f \in \cD(Y)$.
  \item $(f \circ R - f \circ S) \in \hat\cN(X)$ for all $f \in \Cinf(Y)$.
 \end{enumerate}
\end{theorem}
\begin{proof}
(i)$\Rightarrow$(ii) is clear.\\
(ii)$\Rightarrow$(iii): Suppose that (iii) is violated. Then there exists some Riemannian metric $h$ on $Y$,
some compact subset $K$ of $X$, and some $m_0\in \bN_0$ such that: $\forall n\in \bN_0$ $\exists \Phi_n\in 
\sk n{X}$ $\forall k\in \bN_0$ $\exists$ $\eps_{nk}$ such that $\eps_{nk}\searrow 0$ $(k\to \infty)$ $\exists
p_{nk}\in K$ with 
\begin{equation}\label{simproof1}
 d_h(R(\Phi_n(\eps_{nk},p_{nk}),p_{nk}),S(\Phi_n(\eps_{nk},p_{nk}),p_{nk}))) > \eps_{nk}^{m_0} \quad \forall k.
\end{equation}
By Definition\ \ref{equivdef} (i) there exists some $N$ such that for all $n\ge N$ we have
\begin{equation}\label{simproof2}
d_h(R(\Phi_n(\eps_{nk},p_{nk}),p_{nk}),S(\Phi_n(\eps_{nk},p_{nk}),p_{nk}))) \to 0 \quad (k\to \infty).
\end{equation}
By assumption, both $R$ and $S$ are c-bounded. Hence there is a compact subset $L$ of $Y$ such that $\forall n\in \N_0$ $\exists \eta_n>0$ $\forall \eps<\eta_n$ $\forall p\in K$: 
$R(\Phi_n(\e,p),p)$, $S(\Phi_n(\e,p),p) \in L$. For each $n\ge N$ choose some $k_n$ such that
$\eps_{nk} < \eta_n$ for all $k\ge k_n$. Then for each $n\ge N$ and each $k\ge k_n$, both
$R(\Phi_n(\eps_{nk},p_{nk}),p_{nk})$ and $S(\Phi_n(\eps_{nk},p_{nk}),p_{nk})$ are elements of $L$.

Since $L$ is compact we may assume, passing to subsequences if necessary, that for each $n\ge N$ 
the sequences $(R(\Phi_n(\eps_{nk},p_{nk}),p_{nk}))_k$, $(S(\Phi_n(\eps_{nk},p_{nk}),$ $p_{nk}))_k$
are convergent. By (\ref{simproof2}) they have the same limit $q_n\in L$. Again by passing to 
a subsequence we may additionally suppose that $q_n \to q\in L$. 

Choose a chart $(W,\psi)$ around $q$ with $\psi(q)=0$ and $\psi(W)=B_r(0)$. For $n$, $k$ sufficiently large, $R(\Phi_n(\eps_{nk},p_{nk}),p_{nk})$, 
$S(\Phi_n(\eps_{nk},p_{nk}),p_{nk})\in \psi^{-1}(B_{r/2}(0))$. Choose a Riemannian metric
$g$ on $Y$ such that $\psi_*(g|_{B_{r/2}})$ is the standard Euclidean metric on $\R^m$. 
Since $R\sim_0 S$ by assumption we conclude from Definition
\ref{equivdef} (ii) that
\begin{multline*}
d_g(R(\Phi_n(\eps_{nk},p_{nk}),p_{nk}),S(\Phi_n(\eps_{nk},p_{nk}),p_{nk}))) \le \\
\|\psi\circ R(\Phi_n(\eps_{nk},p_{nk}),p_{nk}) - \psi\circ S(\Phi_n(\eps_{nk},p_{nk}),p_{nk})\| 
\le \eps_{nk}^{2m_0} 
\end{multline*}
for $n$, $k$ sufficiently large. Since for some $C>0$, $d_h(q_1,q_2) \le C d_g(q_1,q_2)$ for all $q_1,\, q_2 \in L$ (cf.\ 
\cite{GKOS}, Lemma 3.2.5), this contradicts (\ref{simproof1}).

(iii)$\Rightarrow$(ii): We first note that (i) of Definition \ref{equivdef} is obvious. To
show (ii) of Definition \ref{equivdef} for $k=0$, let $L\subset\subset V$ for $(V,\varphi)$ some
chart in $X$ and $L'\subset\subset W$ for $(W,\psi)$ a chart in $Y$. Let us first assume
that $L'$ is contained in a convex (in the sense of \cite{ON}) set $W'$ with 
$\overline{W'} \subset\subset W$. Let $m\in \N$ and choose an $N$ suitable for $L$ and $m$ 
according to (iii). Let $\Phi\in \sk N{X}$ and $\eps_0>0$ such that for all $\eps<\eps_0$ and all $p\in L$
$$
d_h(R(\Phi(\e,p),p), S(\Phi(\e,p),p)) < \eps^m\,.
$$
Now let $\eps<\eps_0$ and $p\in L$ be such that $R(\Phi(\eps,p),p)$,
$S(\Phi(\eps,p),p)$ $\in$ $L'$. By convexity,
$$
d_h(R(\Phi(\e,p),p), S(\Phi(\e,p),p)) = \int_0^1 \|\gamma_\eps'(s)\|_h\,ds\,,
$$ 
where $\gamma_\eps:[0,1] \to W'$ is the unique geodesic in $W'$ connecting
$R(\Phi(\eps,p),p)$ and $S(\Phi(\eps,p),p)$. Since $W'$ is relatively compact there exists
some $C>0$ such that $\|\xi\| \le C \|T_{\psi(q)}\psi^{-1}(\xi)\|_h$ for all $q\in W'$
and all $\xi\in \R^m$ (with $\norm{\ }$ the Euclidean norm on $\R^m$). Thus
\begin{eqnarray*}
\|\psi\circ R(\Phi(\eps,p),p) - \psi\circ S(\Phi(\eps,p),p)\| &\le& \int_0^1\|(\psi\circ\gamma_\eps)'(s)\|\,ds \\
&\le& C \int_0^1 \|\gamma_\eps'(s)\|_h\,ds\,,
\end{eqnarray*}
which gives the result in this case. 

For general $L'$ we write $L'=\bigcup_{i=1}^k L_i'$ with $L_i'\comp W_i$,  $W_i$ convex
and $\overline{W_i} \comp W$ for all $i$. Given $m\in \N$ let $N$ be as in (iii) and
let $\Phi\in \sk N{X}$. Pick $\eps_0>0$ such that for all $p\in L$ and all $\eps<\eps_0$
$$
d_h(R(\Phi(\e,p),p), S(\Phi(\e,p),p)) < \min_{1\le i\le k} d_h(L_i',\partial W_i)\,.
$$
Then if $\eps<\eps_0$ and $p\in L$ is such that $R(\Phi(\e,p),p)$, $S(\Phi(\e,p),p))\in L'$
there exists an $i\in \{1,\dots,k\}$ with $R(\Phi(\e,p),p)$, $S(\Phi(\e,p),p))\in W_i$.
By what has been shown above this entails the claim.

(ii)$\Rightarrow$(iv): By \cite{GKSV}, Cor.\ 4.5 it suffices to show that, given any
$f\in \cD(Y)$, for each $K\comp X$ and each $m\in \N$ $\exists N\in \N_0$ $\forall \Phi\in \sk N{X}$
$\exists \e_0>0$ $\forall \e<\e_0$ $\forall p\in K$: 
$$
|f\circ R(\Phi(\eps,p),p) - f\circ S(\Phi(\eps,p),p)| < \eps^m\,.
$$
To this end we first observe that since $R$, $S$ are c-bounded there exists some $L\comp Y$
such that for all $\Phi\in \sk0X$ 
$\exists \eps_0>0$ $\forall \eps<\eps_0$ $\forall p\in K$: $R(\Phi(\eps,p),p)$, $S(\Phi(\eps,p),p) 
\in L$. We cover $L$ by open sets $W_l'$ such that $\overline{W_l'} \comp W_l$,
where $(W_l,\psi_l)$ is a chart in $Y$ for $l=1,\dots,s$.

By (i) in Definition \ref{equivdef} there is $N \in \N_0$ such that for given $\Phi\in \sk N{X}$ we may assume
$\eps_0$ to be so small that for any $p\in K$ and any $\eps<\eps_0$ there is an $l\in \{1,\dots,s\}$
with $R(\Phi(\eps,p),p)$, $S(\Phi(\eps,p),p)\in W_l'$ (this follows as in (ii)$\Rightarrow$(iii)).
Thus,
\begin{multline*}
|f\circ R(\Phi(\eps,p),p) - f\circ S(\Phi(\eps,p),p)| = \\
|(f\circ \psi_l^{-1})\circ \psi_l\circ R(\Phi(\eps,p),p) - (f\circ \psi_l^{-1})\circ \psi_l \circ S(\Phi(\eps,p),p)|\,.
\end{multline*}
By \cite{kun}, Lemma 2.5 there exists a constant $C>0$ (depending exclusively
on $\psi_l$, $f$ and $L$) such that this last expression
can be estimated by $\|\psi_l\circ R(\Phi(\eps,p),p) -  \psi_l \circ S(\Phi(\eps,p),p)\|$.
By (ii) from Definition \ref{equivdef} this concludes this part of the proof.

(iv)$\Rightarrow$(i): We first show (i) from Definition \ref{equivdef}. Using a Whitney-embedding
we may suppose that $Y\subseteq \R^{m'}$. Let $K\subset\subset X$. Since $R$ and $S$ are c-bounded, 
there exists some $L\subset\subset Y$ such that for all $\Phi\in \sk0{X}$ 
$\exists \eps_0>0$ $\forall \eps<\eps_0$ $\forall p\in K$: $R(\Phi(\eps,p),p)$, $S(\Phi(\eps,p),p) 
\in L$. 

Fix $i\in {1,\dots,m'}$. Denoting by $\mathrm{pr}_i: \R^{m'} \to \R$ the $i$-th projection, pick some
$f\in \cD(Y)$ such that $f=\mathrm{pr}_i$ in a neighborhood of $L$. Applying (iv) to this $f$ we obtain
the existence of some $N_i\in \N_0$ such that for all $\Phi \in \sk {N_i}{X}$ $\exists 0<\eps_i<\eps_0$ $\forall
\eps<\eps_i$ $\forall p\in K$:
$$
|\mathrm{pr}_i\circ R(\Phi(\eps,p),p) - \mathrm{pr}_i\circ S(\Phi(\eps,p),p)| < \eps\,.
$$
Setting $\bar N:=\max \{N_i\mid 1\le i \le m'\}$ this implies that for any $\Phi \in \sk {\bar N}{X}$
and any $\delta>0$ there exists some $\bar \eps > 0$ such that $\forall \eps<
\bar\eps$ $\forall p\in K$: 
\begin{equation}\label{simproof3}
\|R(\Phi(\eps,p),p) -  S(\Phi(\eps,p),p)\| < \delta\,,
\end{equation}
with $\|\,\|$ the Euclidean norm on $\R^{m'}$. Denoting by $g$ the Riemannian metric on $Y$ induced
by the standard Euclidean metric on $\R^{m'}$, it follows from \cite{V}, Lemma A.1 that there exists
a constant $C>0$ such that $d_g(q_1,q_2)\le C\|q_1 - q_2\|$ for all $q_1,\, q_2 \in L$.  Since 
$d_h$, in turn, can be estimated by $d_g$ on $L$ (as in (ii)$\Rightarrow$(iii) above), (\ref{simproof3})
therefore implies (i) from Definition \ref{equivdef}.

To also show (ii) from that definition, let $L$ be a compact subset of some chart $(V,\varphi)$ in $X$ and
$L'\subset\subset W$ for some chart $(W,\psi)$ in $Y$. Let $j\in \{1,\dots,m\}$ and choose $f_j\in \cD(Y)$
such that $f_j=\psi_j$ in a neighborhood of $L'$. By (iv), $f_j\circ R - f_j\circ S \in \hat\cN(X)$. 
Therefore, given $k$ and $m$ in $\N_0$ there exists some $N_j\in \N_0$ such that for all $\tilde\phi\in \lsk
{N_j}{\varphi(V)}$ there is some $\eps_j>0$ such that for all $\eps<\eps_j$ and all $x\in \varphi(L)$,
$$
\|D^{(k)}(f_j\circ R_V(\tilde\phi(\eps,x),x)) - D^{(k)}(f_j\circ S_V(\tilde\phi(\eps,x),x))\| \le \eps^m\,.
$$
For $N:=\max \{N_j\mid 1\le j \le m\}$ and due to our choice of $f_j$, this implies the claim.  

(v)$\Rightarrow$(iv) is obvious.

(iv)$\Rightarrow$(v): Using c-boundedness of $R$ and $S$ this immediately follows by multiplying
any given $f\in \cC^\infty(Y)$ with a suitable compactly supported cut-off function.
\end{proof}

As was already indicated in the proofs of Proposition \ref{boundchar} and Theorem \ref{equivchar}, 
it is sometimes advantageous to view the
target manifold as embedded into some ambient $\R^{m'}$. We next analyze this
situation in some more detail (cf.\ \cite{sheaves} for a corresponding discussion in
the special Colombeau setting). 
\begin{definition} \label{embeddef}
Let $Y$ be a (regular) submanifold of $\R^{m'}$. Then $\mvqt{X}{Y}$ is defined to
be the subspace of $\hat \cG(X)^{m'}$ consisting of those elements that possess 
a representative $R$ satisfying
\begin{itemize}
\item[(i)] $R(\ub0X \times X)\subseteq Y$.
\item[(ii)]  $\forall K \csub X$ $\exists L \csub Y$ $\forall \Phi \in \sk0{X}$
 $\exists \e_0>0$ $\forall \e<\e_0$ $\forall p \in K$ $R(\Phi(\e,p),p) \in L$.
\end{itemize}
\end{definition}
Thus we want the representative to map $X$ into $Y$ and to be c-bounded. The next result shows
that for $Y$ embedded into $\R^{m'}$, $\mvqt{X}{Y}$ is indeed isomorphic to $\mvq{X}{Y}$.
\begin{proposition}
Let $i:Y\hookrightarrow \R^{m'}$ be an embedding. Then the map $i_*: \mvq{X}{Y} \to \mvqt{X}{i(Y)}$,
$i_*(R) = i\circ R$ is a bijection that commutes with restriction to open subsets. 
In particular, if $Y$ is a regular submanifold of $\R^{m'}$ we may 
identify $\mvq{X}{Y}$ and $\mvqt{X}{Y}$.
\end{proposition}
\begin{proof}
Let $R\in \mvm{X}{Y}$. Then by Proposition \ref{modchar}, for each $1\le j\le m'$, $i_j\circ R \in 
\hat \cE(X)$. Moreover, (i) and (ii) from Definition \ref{embeddef} are obviously satisfied. 
If $R$, $S\in \mvm{X}{Y}$ and $R\sim S$ then Theorem \ref{equivchar} implies that
$(i_j\circ R - i_j\circ S) \in \hat \cN(X)$ for each $j$. Thus $i_*(R) = i_*(S)$ in
$\hat \cG(X)^{m'}$, so $i_*$ is well-defined. It is clear from the definitions that
$i_*$ commutes with restrictions to open sets. 

Suppose that $[i_*R] = [i_*S]$ for $R$, $S\in \mvm{X}{Y}$. Then for each $j$, 
$(i_j\circ R - i_j\circ S)\in \hat \cN(X)$. Using the Riemannian metric $i^*g$
induced on $Y$ by the standard Euclidean metric on $\R^{m'}$ it now follows
exactly as in the proof of Theorem \ref{equivchar}, (iv)$\Rightarrow$(i) that
$R\sim S$. Hence $i_*$ is injective. 

Finally, to see that $i_*$ is surjective, suppose that $R\in \hat \cE_M(X)^{m'}$
satisfies (i) and (ii) from Definition \ref{embeddef}. Then by Proposition \ref{modchar},
$i^{-1}\circ R$ is an element of $\mvm{X}{Y}$ whose image under $i_*$ is $R$.
\end{proof}

Finally, we note that the space ${\mathcal C}^\infty(X,Y)$ can naturally be embedded into $\mvq{X}{Y}$
via the map
\begin{eqnarray*}
\sigma: {\mathcal C}^\infty(X,Y) &\hookrightarrow&  \mvq{X}{Y} \\
f &\mapsto& [(\omega,p) \mapsto f(p)].
\end{eqnarray*}

\section{Generalized vector bundle homomorphisms} \label{Hom}

As a natural next step in the development of a theory of manifold-valued generalized functions 
we now introduce a suitable notion of generalized vector bundle homomorphisms. Again
we take our motivation for the concrete form of the definitions below from the case
of special Colombeau algebras (\cite{kun,kun2}).

\begin{definition}
  $\vmb{E}{F}$ is the set of all $s \in \Cinf(\sk0X \times E, F)$ such that $s(\omega, \cdot) \in \Hom(E,F)$ for each $\omega \in \sk0X$.
\end{definition}

The appropriate notions of moderateness and negligibility are as follows. 

\begin{definition}\label{vbmod}
 $s \in \vmb{E}{F}$ is called \emph{moderate} if
\begin{enumerate}[(i)]
 \item $\underline{s} \in \mvm{X}{Y}$,

 \item for all vector bundle charts $(V,\Phi)$ in $E$ and $(W,\Psi)$ in $F$, all $L \csub V$ and $L' \csub W$ and all $k \in \bN_0$ there exists $N \in \bN$ such that for all $\tilde\phi \in \lsk0{\varphi(V)}$
there exists some $\eps_0>0$ such that, for all $\eps<\eps_0$,
$\norm{D^{(k)} (s_{W,V}^{(2)}(\tilde\phi(\e,x),x)) }  \le\e^{-N}$, uniformly for $x \in \varphi(L) \cap \underline{s}_{W,V}(\tilde\phi(\e, \cdot), \cdot)^{-1}(\psi(L'))$.
\end{enumerate}
By $\vmm{X}{Y}$ we denote the space of all moderate elements of $\vmb{X}{Y}$.
\end{definition}

\begin{definition}\label{vbneg}
 Two elements $s,t \in \vmm{E}{F}$ are called \emph{vb-equivalent} (denoted by $s \vbsim t$) if

\begin{enumerate}[(i)]
 \item $\underline{s} \sim \underline{t}$ in $\mvm{X}{Y}$,
 \item for all vector bundle charts $(V,\Phi)$ in $E$ and $(W,\Psi)$ in $F$, all $L \csub V$ and $L' \csub W$ and all $k \in \bN_0$ and $m \in \bN$ there exists $q \in \bN_0$ such that for all $\tilde\phi \in \lsk q{\varphi(V)}$ there exists
 $\eps_0>0$ such that for all $\eps<\eps_0$: 
$$
\norm{D^{(k)} ((s_{W,V}^{(2)}-t_{W,V}^{(2)})(\tilde\phi(\e,x),x)) } \le \e^{-N}$$
uniformly for $x \in \varphi(L) \cap \underline{s}_{W,V}(\tilde\phi(\e, \cdot), \cdot)^{-1}(\psi(L')) \cap \underline{t}_{W,V}(\tilde\phi(\e, \cdot), \cdot)^{-1}(\psi(L'))$.
\end{enumerate}
\end{definition}

By $\vbzsim$ we denote the corresponding relation where (ii) only holds for $k=0$.

\begin{definition}$\vmq{E}{F} \coleq \vmb{E}{F} / \vbsim$ is the space of generalized vector bundle homomorphisms from $E$ to $F$.
\end{definition}

Consider $s \in \vmb{E}{\R\times\R^{m'}}$ and let $(V, \Phi)$ be a vector bundle chart in $E$. From $s_V(\phi, x, \xi) = (s_V^{(1)}(\phi, x), s_V^{(2)}(\phi,x) \cdot \xi)$ one can directly read off the following characterization of moderateness and negligibility for this simple form of the range space.

\begin{lemma}\label{daslemma}Let $s \in \vmb{E}{\R\times\R^{m'}}$.
\begin{enumerate}[(i)]
 \item $s \in \vmm{E}{\R\times\R^{m'}}$ if and only if $\underline{s}$ is c-bounded and for each vector bundle chart $(V, \Phi)$ in $E$, $s_V^{(1)} \in \cE_M(\varphi(V))$ and $s_V^{(2)} \in \cE_M(\varphi(V))^{m'\cdot n'}$.
 \item For $s,t \in \vmm{E}{\R\times\R^{m'}}$, $s \vbsim t$ if and only if $s_V^{(1)} - t_V^{(1)} \in \cN(\varphi(V))$ and $s_V^{(2)} - t_V^{(2)} \in \cN(\varphi(V))$ for each vector bundle chart $(V, \Phi)$ in $E$.
\end{enumerate}
\end{lemma}

From (ii) and \cite[Theorem 2.5.4]{GKOS}
it follows that $\vbsim$ and $\vbzsim$ coincide on $\vmb{E}{\R \times \R^{m'}}$. Moreover, by \cite[Theorems 3.3.15 and 3.3.16]{GKOS} one can replace ``$s_V^{(1)} \in \cE_M(\varphi(V))$'' by ``$\underline{s} \in \hat\cE_M(X)$'' and ``$s_V^{(1)} - t_V^{(1)} \in \cN(\varphi(V))$'' by ``$\underline{s} - \underline{t} \in \hat\cN(X)$'', respectively, in Lemma \ref{daslemma}.

Next, we derive some intrinsic characterizations of the spaces just defined.

\begin{proposition}\label{vbchar}
 Let $s \in \vmb{E}{F}$. The following statements are equivalent:
\begin{enumerate}[(a)]
 \item $s \in \vmm{E}{F}$.
 \item \begin{enumerate}[(i)]
    \item $\underline{s}$ is c-bounded,
    \item $\hat f \circ s \in \vmm{E}{\bR \times \bR^{m'}}$ for all $\hat f \in \Hom_c(F, \bR \times \bR^{m'})$.
    \end{enumerate}
 \item \begin{enumerate}[(i)]
\item $\underline{s}$ is c-bounded,
\item $\hat f \circ s \in \vmm{E}{\bR \times \bR^{m'}}$ for all $\hat f \in \Hom(F, \bR \times \bR^{m'})$.
\end{enumerate}
\end{enumerate}

\end{proposition}

\begin{proof}
(a) $\Rightarrow$ (c): We first note that $\underline{s}$ is c-bounded by Definition \ref{vbmod}. For (ii), by Lemma \ref{daslemma} we first have to show that $\underline{\hat f \circ s}$ is c-bounded and an element of $\hat\cE_M(X)$. Because $\underline{s} \in \mvm{X}{Y}$, Proposition \ref{boundchar} implies c-boundedness of $\underline{\hat f \circ s} = \underline{\hat f} \circ \underline{s}$, while Proposition \ref{modchar} implies its moderateness.

It remains to show that for each vector bundle chart $(V, \Phi)$ in $E$, $(\hat f \circ s)_V^{(2)}$ is in $\cE_M(\varphi(V))^{m' \cdot n'}$. For the moderateness test let $\tilde \phi \in \lsk0{\varphi(V))}$, $k \in \bN_0$ and $L \csub V$ be given. Choose $L' \csub Y$ such that $\forall \Phi \in \sk0X$ $\exists \e_0>0$ $\forall \e<\e_0$ $\forall p \in L$: $\underbar{s}(\Phi(\e,p),p) \in L'$. Cover $L'$ by vector bundle charts $(W_l, \Psi_l)$ in $F$ and write $L' = \bigcup L_l'$ with $L_l' \csub W_l$.
For $\e<\e_0$ we then have
\begin{equation}\label{eq3}
L = \bigcup_l L \cap (\underbar{s}((\varphi^* \circ \hat\lambda^{-1})(\tilde \phi(\e,\varphi(\cdot)),\cdot)))^{-1}(W_l).
\end{equation}

For any $(\phi, x) \in ((\hat\lambda \circ \varphi_*) \times \varphi)((\ub0V \times V) \cap \underline{s}^{-1}(W_l))$ we can write
\begin{equation}\label{eq1}
(\hat f \circ s)_V^{(2)} (\phi, x) = \hat f_{W_l}^{(2)} ( s_{W_l,V}^{(1)}(\phi, x)) \cdot s_{W_l,V}^{(2)}(\phi, x).
\end{equation}

In particular, \eqref{eq1} holds for the pair $(\tilde\phi(\e,x),x)$ with $x \in  (\underline{s}((\varphi^* \circ \hat\lambda^{-1})(\tilde \phi(\e, \cdot)), \varphi^{-1}(\cdot)))^{-1}(W_l)$. Because the latter is an open set this even holds for all $x'$ in a neighborhood of $x$. In order to obtain the required moderateness estimate we have to estimate derivatives of $(\hat f \circ s)_V^{(2)}(\phi(\e,x),x)$ uniformly for $x \in \varphi(L)$. Now by the decomposition \eqref{eq3} of $L$, on each of the sets $\varphi(L \cap (\underline{s}((\varphi^* \circ \hat\lambda^{-1})(\tilde\phi(\e,\varphi(\cdot)), \cdot))^{-1}(W_l))$ we can use equation \eqref{eq1}, by which the claim follows from moderateness of $\hat f_{W_l}^{(2)}$, $s_{W_l, V}^{(2)}$ and c-boundedness of $s^{(1)}$.

(c) $\Rightarrow$ (b) is clear.

(b) $\Rightarrow$ (a): We first have to show that $\underline{s}$ is moderate. Using Proposition \ref{modchar} we have to establish that, given any $f \in \cD(Y)$, $f \circ \underline{s}$ is moderate. For this we choose any compactly supported vector bundle homomorphism $\hat f \in \Hom_c(E,F)$ such that $\underline{\hat f} = f$. Then $\hat f \circ s$ is moderate by assumption, which implies that $\underline{\hat f \circ s} = f \circ \underline{s}$ is moderate.

For the second part, take vector bundle charts $(V, \Phi)$ in $E$ and $(W, \Psi)$ in $F$, $L \csub V$ and $L' \csub W$. Choose an open, relatively compact neighborhood $U$ of $L'$ in $W$. For any $l \in \{\,1,\ldots,m\,\}$ choose $\hat f_l \in \Hom_c(F, \R \times \R^{m'})$ such that

\[ 
\hat f_l|_{\pi_F^{-1}(U)} = (\pr_l \times \id_{\R^{m'}}) \circ \Psi|_{\pi_F^{-1}(U)}. 
\]

Let $x \in \varphi(L) \cap \underline{s}_{W,V}(\tilde \phi(\e,\cdot),\cdot)^{-1}(\psi(L'))$. Then for $x'$ in a suitable neighborhood of $x$ we have
\begin{equation}\label{eq2}
 (\hat f_l \circ s)_V^{(2)}(\tilde\phi(\e,x),x) =  s_{W,V}^{(2)}(\tilde\phi(\e,x),x)
\end{equation}
from which the desired estimates follow.
\end{proof}

\begin{theorem} Let $s,t \in \vmm{E}{F}$. The following statements are equivalent:
 \begin{enumerate}[(i)]
  \item $s \vbsim t$.
  \item $s \vbzsim t$.
  \item $\hat f \circ s \vbsim \hat f \circ t$ in $\vmm{E}{\bR \times \bR^{m'}}$ for all $\hat f \in \Hom_c(F, \bR \times \bR^{m'})$.
  \item $\hat f \circ s \vbsim \hat f \circ t$ in $\vmm{E}{\bR \times \bR^{m'}}$ for all $\hat f \in \Hom(F, \bR \times \bR^{m'})$.
 \end{enumerate}
\end{theorem}
\begin{proof}
(i) $\Rightarrow$ (ii) is clear.

(ii) $\Rightarrow$ (iv): let $\hat f \in \Hom(F, \bR \times \bR^{m'})$. By Lemma \ref{daslemma} we have to establish negligibility estimates of order zero for $(\hat f \circ s)_V^{(1)} - (\hat f \circ t)_V^{(1)}$ (equivalently, for $\underline{\hat f \circ s} - \underline{\hat f \circ t}$) and $(\hat f \circ s)_V^{(2)} - (\hat f \circ t)_V^{(2)}$. Fix $L \csub V$ for testing and choose $L' \csub Y$ such that $\forall \Phi \in \sk0X$ $\exists \e_0>0$ $\forall \e_0<0$ $\forall p \in L$: $\underline{s}(\Phi(\e,p),p)) \in L'$ and $\underline{t}(\Phi(\e,p),p)) \in L'$. Cover $L'$ by vector bundle charts $(W_l, \Psi_l)$ in $F$ and choose open sets $L_l'$ with $\overline{L_l'} \csub W_l$ and $L' \subseteq \bigcup_l L_l'$. 

By Theorem \ref{equivchar} for every Riemannian metric $h$ on $Y$ $\exists N \in \bN_0$ such that $\forall \Phi \in \sk{N}{X}$ $\sup_{p \in L}d_h(\underline{s}(\Phi(\e,p),p), \underline{t}(\Phi(\e,p),p))$ converges to $0$ when $\e \to 0$, hence there exists some $\eps_1<\eps_0$ such that for all
 $\e\le\e_1$ and $p \in L$ there exists $l$ such that both $\underline{s}(\Phi(\e,p),p)$ and $\underline{t}(\Phi(\e,p),p)$ are contained in $L_l'$. Now for the first estimate, $s \vbzsim t$ implies $\underline{s} \sim \underline{t}$ by definition, and thus $\underline{\hat f \circ s} - \underline{\hat f \circ t} = \underline{\hat f} \circ \underline{s} - \underline{\hat f} \circ \underline{t} \in \hat\cN(X)$ is implied by Theorem \ref{equivchar}. 

For the second estimate, the norm of
\[ (\hat f \circ s)_V^{(2)}(\tilde \phi(\e,x), x) - (\hat f \circ t)_V^{(2)}(\tilde \phi(\e,x),x) \]
needs to be estimated on $\varphi(L)$ and given $m \in \bN$ it should have, for some $q$, growth of $O(\e^m)$ for all $\tilde \phi \in \lsk{q}{\varphi(V)}$. But this follows from the assumptions using a construction identical to that of (a) $\Rightarrow$ (c) of the previous proposition and \cite[Lemma 2.5]{kun}.

(iv) $\Rightarrow$ (iii) is clear.

(iii) $\Rightarrow$ (i) Given any $f \in \cD(Y)$, choose $\hat f \in \Hom_c(F, \R \times \R^{m'})$ with $\underline{\hat f} = f$; from the assumption we have $f \circ \underline{s} \sim f \circ \underline{t}$ which together with Theorem \ref{equivchar} gives (i) of Definition \ref{vbneg}. A construction as in the proof of Proposition \ref{vbchar} (a) $\Rightarrow$ (c) employing a representation of both $s_{W,V}^{(2)}$ and $t_{W,V}^{(2)}$ as in equation \eqref{eq2} immediately gives (ii) of Definition \ref{vbneg}.
\end{proof}

As in the case of manifold-valued generalized functions also for vector bundle homomorphisms we have a 
natural embedding of smooth maps, again denoted by $\sigma$:
\begin{eqnarray*}
\sigma: \Hom(E,F) &\hookrightarrow&  \vmq{E}{F} \\
s &\mapsto& [(\omega,p) \mapsto s(p)].
\end{eqnarray*}

Based on these notions we may now define the tangent map of any $[R]\in \mvq{X}{Y}$ 
as an appropriate generalized vector bundle homomorphism:
\begin{definition}
For any $[R] \in \mvq{X}{Y}$ we define the tangent map of $[R]$ as
the class of $(\omega,p) \mapsto \tang_p(R(\omega, \cdot))$ in $\vmq{TX}{TY}$.
\end{definition}
It is immediate from the definitions that this gives a well-defined map.
Moreover, for $f\in {\mathcal C}^\infty(X,Y)$ and $R=\sigma(f)$ it follows that 
$T(\sigma(f)) = \sigma(Tf)$. To see this it suffices to note that $\sigma$ locally
commutes with derivation by \cite[Sec.\ 5]{GKSV}.

\vskip1em

\noindent{\bf Acknowledgment} This work was supported by projects P20525, P23714 and Y237 of the Austrian Science Fund FWF.

\end{document}